\documentclass[a4paper,12pt]{amsart}
\usepackage{amsmath,latexsym,amsfonts,amssymb,mathrsfs,amsthm,hyperref, enumitem}
\usepackage[T1]{fontenc}

\title{Subdifferential set of the joint numerical radius of a tuple of matrices}
\author[Grover]{Priyanka Grover}
\author[Singla]{Sushil Singla}
\address{Department of Mathematics, Shiv Nadar University, NH-91, Tehsil Dadri, Gautam Buddha Nagar, U.P. 201314, India.}
\email{priyanka.grover@snu.edu.in, ss774@snu.edu.in}

\newcommand{\C}{\mathbb{C}}
\newcommand{\R}{\mathbb{R}}
\newcommand{\tr}{\mathop{{\rm trace}}}
\newcommand{\conv}{\mathop{{\rm co}}}
\newcommand{\h}{\mathop{{\rm Re}}}

\newcommand{\hc}{{\mathcal H}}
\newcommand{\hk}{{\mathcal K}}

\newcommand{\dist}{{\mathop{\rm dist}}}

\newtheorem{theorem}{Theorem}

\theoremstyle{definition}

\newtheorem{proposition}{Proposition}

\newcommand{\Hil}{\mathcal{H}}

\newcommand{\tens}[1]{\mathbin{\mathop{\otimes}\limits_{#1}}}
\newcommand{\bra}{\langle}
\newcommand{\ket}{\rangle}

\begin{document}
\begin{abstract}
An expression for the subdifferential set of the joint numerical radius of a tuple of matrices is obtained. Its applications to the Birkhoff-James orthogonality in the space consisting of tuples of matrices equipped with the joint numerical radius norm are discussed.

\end{abstract}

\setlength{\parindent}{0pt}
\setlength{\parskip}{1.6ex}

\maketitle

\numberwithin{theorem}{section}
\numberwithin{corollary}{section}
\numberwithin{equation}{section}
\numberwithin{defn}{section}
\numberwithin{proposition}{section}

\section{Introduction}
Let $\mathbb M_{p,q}(\mathbb C)$ be the normed linear space of $p\times q$ matrices over $\mathbb C$ with any given norm. Let $f: \mathbb M_{p,q}(\mathbb C)\rightarrow \R$ be a continuous convex function. Let $A \in \mathbb M_{p,q}(\mathbb C)$. We denote the conjugate transpose of $A$ by $A^*$. The \emph{subdifferential set} of $f$ at $A$, denoted by $\partial f(A)$, is defined as
$$\partial f(A)=\{C\in\mathbb M_{p,q}(\mathbb C): f(B)\geq f(A)+\h\tr((B-A)^*C) \text{ for all } B\in\mathbb M_{p,q}(\mathbb C)\}.$$
The right hand derivative of $f$ and the subdifferential of $f$ are related as follows (see  \cite[Theorem 1.2.9]{thesis}). For $B\in\mathbb M_{p,q}(\mathbb C)$,
 \begin{equation}\label{chap5:7898}\lim\limits_{t\rightarrow 0^+} \dfrac{f(A+tB) - f(A)}{t} = \max\{\h\tr(C^*B) : C\in \partial f(A)\}.\end{equation} 
Characterizations of subdifferentials of matrix norms have been of interest to many mathematicians. We denote $\mathbb M_{n,n}(\mathbb C)$ by $\mathbb M_{n}(\mathbb C)$. A norm $|||\cdot|||$ on $\mathbb M_{n}(\mathbb C)$ is said to be a \emph{unitarily invariant norm}, if for any unitary matrices $U,U'\in \mathbb M_n(\mathbb C)$ and any $A\in \mathbb M_n(\mathbb C)$, we have $|||UAU'|||=|||A|||$. For $A\in\mathbb M_n(\mathbb C)$, let  $s_1(A)\geq \cdots\geq s_n(A)$ be the singular values of $A$ in decreasing order. Given a  {unitarily invariant norm} $|||\cdot|||$ on $\mathbb M_n(\mathbb C)$, there is a unique \emph{symmetric gauge function} $\Phi$ on $\mathbb R^n$ such that $|||A|||=\Phi((s_1(A), \ldots, s_n(A))$ for every $A\in \mathbb M_n(\mathbb C)$. In \cite[Theorem 2]{Watson}, an expression of $\partial |||\cdot|||$ was given in  $\mathbb M_n(\mathbb R)$. In \cite[Theorem 3.1, Theorem 3.2]{Zietak}, it was shown that for $A\in\mathbb M_n(\mathbb C)$,
\begin{align*}\partial |||A||| = \{& U \text{diag}(d_1,\ldots,d_n) U'^*: A=U \Sigma U'^* \text{ is a singular value decomposition  of A,}\\
&  s_1(A) d_1+\cdots+s_n(A) d_n=|||A|||, \Phi^*((d_1,\ldots, d_n))=1\},\end{align*}
(see also \cite[Theorem 1]{watson93}). 
In \cite{Watson}, the below norm was also considered:
$$\|A\|_{\alpha,\beta}=\max\{\|Ax\|_\alpha:\|x\|_\beta=1\},$$ where $\|\cdot\|_\alpha, \|\cdot\|_\beta$ are any two norms on $\mathbb R^n$. An expression for $\partial \|\cdot\|_{\alpha,\beta}$ was obtained in \cite[Theorem 4]{Watson}.
When $\|\cdot\|_\alpha$ is the usual Euclidean norm on $\mathbb R^n$, the above norm is the \emph{operator norm} (or the \emph{spectral norm}), denoted by $\|\cdot\|$.  %defined as $$\|A\|=\max\{\|Au\|: \|u\|=1\}.$$ The operator norm is also a unitarily invariant norm. 
For $A\in M_n(\mathbb C)$, 
\begin{equation}\partial \|A\|=\conv\{uv^*:u,v\in \mathbb C^n, \|u\|=\|v\|=1, Av=\|A\|u\},\label{90}\end{equation} where $\conv(S)$ denotes the convex hull of a set $S$. See \cite[Example 3]{Watson} for a proof of \eqref{90} when $A\in \mathbb M_n(\mathbb R)$. A similar and straightforward calculation shows \eqref{90} for $A\in \mathbb M_n(\mathbb C)$. For $1\leq k\leq n$, the Ky Fan $k$-norm $\|\cdot\|_{(k)}$ is defined as $$\|A\|_{(k)}=s_1(A)+\dots +s_k(A).$$
The subdifferential set of the {Ky Fan $k$-norms} on $\mathbb M_n(\mathbb C)$ was obtained in \cite[Theorem 2.7]{2017}. Another useful norm on $\mathbb M_n(\mathbb C)$ is the \emph{numerical radius}, defined as
$$\omega(A)=\max_{\|x\|=1} \left|\langle x, Ax\rangle\right|.$$ This paper is devoted to the study of the subdifferential set of $\omega(\cdot)$. More generally, we consider the \emph{joint numerical radius} of a tuple of matrices defined as follows.
Let $A_1, \dots, A_d\in \mathbb M_n(\mathbb C)$. Let $\boldsymbol{A}=(A_1, \dots, A_d):\mathbb C^n \rightarrow (\mathbb C^n)^d$ be defined as $\boldsymbol{A}x=(A_1x, \dots, A_dx)$ for all $x\in\mathbb C^n$. The {joint numerical radius} of $\boldsymbol{A}$ is defined as 
$$\omega(\boldsymbol{A})=\max_{x\in \mathbb C^n,\ \|x\|=1}\left(\sum\limits_{k=1}^d\big|\langle x| A_kx\rangle\big|^2\right)^{1/2}.$$
%Note that the joint numerical radius $\omega(\cdot)$ is a norm on $\mathbb M_n(\mathbb C)^d$.  
%The main results of the paper are given below. Few notations are in order. 
%For a subset $S$ of $\mathbb M_{p,q}(\mathbb C)$, let $\conv(S)$ denote the closed convex hull of $S$. 
For $x\in\mathbb C^n$, let $x\bar{\tens{}} x$ be the rank one operator on $\mathbb C^n$ defined as $x\bar{\tens{}} x(y) = \bra x | y\ket x$ for all $y\in\C^n$. We will use the same symbol $x\bar{\tens{}} x$ for the rank one operator as well as its matrix representation with respect to the standard basis of $\mathbb C^n$. Note that $(x\bar{\tens{}} x)^* = x\bar{\tens{}} x$ and for any $C\in\mathbb M_n(\mathbb C)$, we have $\tr\left((x\bar{\tens{}} x)C\right) = \langle x|Cx\rangle$. Let $\boldsymbol{0}=(0,\dots, 0)\in\mathbb M_n(\C)^d$. The main result of this paper is as follows.
\begin{theorem}\label{main}
Let $\boldsymbol{A}\in\mathbb M_n(\C)^d\setminus\{\boldsymbol{0}\}$. Then
\begin{enumerate}
\item[(i)] the subdifferential set of $\omega(\cdot)$ at $\boldsymbol{A}$ is given by
\begin{align}
\partial \omega(\boldsymbol{A})=\conv\bigg\{&\frac{1}{\omega(\boldsymbol{A})} \left(\langle x|A_1x\rangle x\bar{\tens{}}x, \dots, \langle x|A_dx\rangle x\bar{\tens{}}x\right)  :  \|x\|=1,\nonumber\\
&\omega(\boldsymbol{A})=\left(\sum\limits_{k=1}^d\big|\langle x|A_kx\rangle\big|^2\right)^{1/2}\bigg\},\label{subdifferential}
\end{align}
and
\item[(ii)] for $\boldsymbol{B}\in\mathbb M_n(\C)^d$,
\begin{equation*}\lim\limits_{t\rightarrow 0^+}\frac{\omega(\boldsymbol{A}+t\boldsymbol{B}) -\omega(\boldsymbol{A})}{t} = \frac{1}{\omega(\boldsymbol{A})}\max_{ \|x\|=1,\ \omega(\boldsymbol{A})=\left(\sum\limits_{k=1}^d|\langle x|A_kx\rangle|^2\right)^{1/2}}\h\sum\limits_{k=1}^d\overline{\langle x|A_kx\rangle}\langle x|B_kx\rangle.\end{equation*}
\end{enumerate}
\end{theorem}

Let $(X,\|\cdot\|)$ be a normed space. An element $x\in X$ is said to be \emph{(Birkhoff-James) orthogonal} to a subspace $W$ in $(X,\|\cdot\|)$ if \begin{equation}\|x+y\|\geq \|x\|\text{ for all }y\in W.\label{synn:new}\end{equation} If $W$ is a one-dimensional subspace generated by $z$ and \eqref{synn:new} is satisfied, then we say that $x$ is orthogonal to $z$. Let $\hc$ be a Hilbert space. Let $\mathscr B(\hc)$ denote the space of bounded linear operators on $\hc$. For $A\in \mathscr B(\hc)$, $\omega(A)=\sup_{\|x\|=1} \left|\langle x, Ax\rangle\right|$. If $A,B\in \mathscr B(\hc)$, then a sufficient condition for $A$ to be orthogonal to $B$ in the space $\left(\mathscr B(\hc), \omega(\cdot)\right)$ follows from  \cite[Theorem 2.11]{arxiv}. As a particular case, it shows that if there exist $h$ unit vectors $x_1,\ldots,x_h \in \mathbb C^n$ with $\omega(A)=\big|\langle x_i|A x_i\rangle\big|$ for all $1\leq i\leq h$ and there exist $h$ positive numbers $t_1,\ldots,t_h>0$ with $t_1+\cdots+t_h=1$ such that $
 \sum\limits_{i=1}^h t_i\overline{\langle x_i|A x_i\rangle}\langle x_i| B x_i\rangle=0,$ then $A$ is  orthogonal to $B$ in $\left(\mathscr B(\hc), \omega(\cdot)\right)$. A characterization of orthogonality in $\left(\mathscr B(\hc), \omega(\cdot)\right)$ has been proved recently in \cite{afterreview}.

For $\boldsymbol{\lambda}=(\lambda_1, \dots, \lambda_d)\in\mathbb C^d$ and $\boldsymbol{B}=(B_1,  \dots, B_d)\in\mathbb M_n(\mathbb C)^d$, let $\boldsymbol{\lambda B} = (\lambda_1 B_1, \dots, \lambda_d B_d)$. As a consequence of Theorem \ref{main},  we obtain characterizations for $\boldsymbol{A}$ to be orthogonal to $\boldsymbol{B}$ in the space $(\mathbb M_n(\mathbb C)^d, \omega(\cdot))$, as well as for $\boldsymbol{A}$ to be  orthogonal to the subspace $\{\boldsymbol{\lambda B}: \boldsymbol{\lambda}\in\C^d\}$.

\begin{theorem}\label{100000}
Let $\boldsymbol{A}=(A_1, \dots, A_d), \boldsymbol{B}=(B_1,  \dots, B_d)\in\mathbb M_n(\mathbb C)^d$. Then $\boldsymbol{A}$ is orthogonal to $\boldsymbol{B}$ in $(\mathbb M_n(\mathbb C)^d, \omega(\cdot))$ if and only if there exist $h$ unit vectors $x_1,\ldots,x_h \in \Hil$ with $\omega(\boldsymbol{A})=\left(\sum\limits_{k=1}^d\big|\langle x_i|A_kx_i\rangle\big|^2\right)^{1/2}$ for all $1\leq i\leq h$ and there exist $h$ positive numbers $t_1,\ldots,t_h>0$ with $t_1+\cdots+t_h=1$ such that $
\sum\limits_{k=1}^d \sum\limits_{i=1}^h t_i\overline{\langle x_i|A_kx_i\rangle}\langle x_i| B_kx_i\rangle=0.$
\end{theorem}

\begin{theorem}\label{101010}
Let $\boldsymbol{A}=(A_1, \dots, A_d)$, $\boldsymbol{B}=(B_1,  \dots, B_d)\in\mathbb M_n(\mathbb C)^d$. Then 
\begin{equation}\label{1111}
\omega(\boldsymbol{A}+\boldsymbol{\lambda B}) \geq \omega(\boldsymbol{A}) \text{ for all } \boldsymbol{\lambda}\in \C^d
\end{equation}
if and only if there exist $h$ unit vectors $x_1,\ldots,x_h \in \C^n$ with $\omega(\boldsymbol{A})=\left(\sum\limits_{k=1}^d\big|\langle x_i|A_kx_i\rangle\big|^2\right)^{1/2}$ for all $1\leq i\leq h$ and there exist $h$ positive numbers $t_1,\ldots,t_h >0$ with $t_1+\cdots+t_h=1$ such that $\sum\limits_{i=1}^h t_i\overline{\langle x_i| A_kx_i\rangle}\langle x_i|B_kx_i\rangle =0$ for all $1\leq k\leq d$.
\end{theorem}
%, we obtain a characterization , \eqref{1111} is equivalent to saying that $\boldsymbol{A}$ is orthogonal to the subspace $\{\boldsymbol{\lambda B}: \boldsymbol{\lambda}\in\C^d\}$ in $\omega(\cdot)$. 

In Section \ref{proofs}, we give the proofs of above theorems. We also obtain some analogous results for the  \emph{joint operator norm}. Finally, we end with some remarks in Section  \ref{Remarks}.

\section{Proofs}\label{proofs}
To prove Theorem \ref{main}, we need the following propositions from the subdifferential calculus. 

%Let $X$ and $Y$ be Banach spaces. Let $\mathscr B(X, Y)$  be the space of bounded $\C$-linear operators from $X$ to $Y$ with operator norm.

\begin{proposition}\label{chainrule} Let $T_1: \mathbb M_{p,q}(\mathbb C)\rightarrow \mathbb M_{r,s}(\mathbb C)$ be a linear map. Let $B\in \mathbb M_{r,s}(\mathbb C)$. Let $T_2: \mathbb M_{p,q}(\mathbb C)\rightarrow \mathbb M_{r,s}(\mathbb C)$ be the  affine map defined as $T_2(A)=T_1(A)+ B$. Let $g: \mathbb M_{r,s}(\mathbb C) \rightarrow \R$ be a continuous convex function. Then for $A \in \mathbb M_{p,q}(\mathbb C)$,
\begin{equation*}
\partial (g\circ T_2)(A)=T_1^* \partial g(T_2(A)).\label{1.2}
\end{equation*}
\end{proposition}

\begin{proposition}\label{maxsubdifferential}
Let $J$ be a compact set in some metric space. Let $\{f_j\}_{j\in J}$ be a collection of continuous convex functions from $\mathbb M_{p,q}(\mathbb C)$ to $\R$ such that for $A\in \mathbb M_{p,q}(\mathbb C)$, the maps $j\rightarrow f_j(A)$ are upper semi-continuous. Let $f: \mathbb M_{p,q}(\mathbb C)\rightarrow \R$ be defined as $f(A)=\sup\{f_j(A):j\in J\}$. Let $J(A)=\{j \in J: f_j(A)=f(A)\}$. Then
\begin{equation*}
\partial f(A)=\conv\left(\cup\left\{\partial f_j(A): j\in J(A)\right\}\right).\label{maxsubdiff}
\end{equation*}
\end{proposition}

The proofs of these can be found in Theorem 4.2.1 and Theorem 4.4.2 of  \cite{hiriart} for real valued convex functions on the Euclidean space $\mathbb R^n$. The same proofs can be extended to real valued continuous convex functions on a normed space (see \cite{thesis, convex} for more detail). %Now we prove Theorem \ref{main}.

\textit{Proof of Theorem \ref{main}.} 
\begin{enumerate}
\item[(i)] In \cite{joint}, it was shown that $\omega(\boldsymbol{A})$ can also be expressed as \begin{equation}\omega(\boldsymbol{A})=\max_{x\in \C^n,\ \|x\|=1}\max_{(\lambda_1, \dots, \lambda_d)\in\C^d, \|(\lambda_1,  \dots, \lambda_d)\|=1}\left| \sum\limits_{k=1}^d \lambda_k\langle x|A_kx\rangle\right|.\label{synk1}\end{equation} 
Let $\boldsymbol{\lambda} = (\lambda_1, \dots, \lambda_d)\in\C^d$ and let $x\in\C^n$. Let $\boldsymbol{C}=(C_1, \dots, C_d)\in\mathbb M_n(\C)^d.$ Let $T_{x,\boldsymbol{\lambda}}: \mathbb M_n(\C)^d\rightarrow \C$ be the linear map defined as
$$T_{x,\boldsymbol{\lambda}}(\boldsymbol{C})=\sum\limits_{k=1}^d \lambda_k\langle x| C_kx\rangle.$$ Let $z\in\C$. Let $g:\C\rightarrow\R$ be the map defined as $g(z)=|z|$.
Let $f_{x,\boldsymbol{\lambda}}: \mathbb M_n(\C)^d \rightarrow \R$ be the map defined as  $$f_{x,\boldsymbol{\lambda}}=g\circ T_{x,\boldsymbol{\lambda}}.$$ %We note that for $\boldsymbol A\neq \boldsymbol{0}$, $f_{x,\boldsymbol{\lambda}}$ is a differentiable map. Hence, $\partial f_{x,\boldsymbol{\lambda}}(\boldsymbol A)$ is the singleton set containing the derivative $\text{D} f_{x,\boldsymbol{\lambda}}(\boldsymbol A)$.  
Let $J$ be the compact set $\left\{(x,\boldsymbol{\lambda})\in \C^n\times\C^d :\ \|x\|=1, \|\boldsymbol{\lambda}\|=1\right\}$. Note that for $\boldsymbol{C}\in \mathbb M_n(\C)^d$, the map $(x,\boldsymbol{\lambda})\rightarrow f_{x,\boldsymbol{\lambda}}(\boldsymbol{C})$ is continuous.  Now \eqref{synk1} can be rewritten as
$$\omega(\boldsymbol{A})=\max\{f_{x,\boldsymbol{\lambda}}(\boldsymbol{A}): (x,\boldsymbol{\lambda})\in J\}.$$ 
Let $J(\boldsymbol{A})=\{(x,\boldsymbol{\lambda})\in J: f_{x,\boldsymbol{\lambda}}(\boldsymbol{A})=\omega(\boldsymbol{A})\}.$ 
By Proposition \ref{maxsubdifferential}, 
$$\partial \omega(\boldsymbol{A})=\conv\left(\cup\left\{\partial f_{x,\boldsymbol{\lambda}}(\boldsymbol{A}): (x,\boldsymbol{\lambda})\in J(\boldsymbol{A})\right\}\right).$$
Let $(x,\boldsymbol{\lambda})\in J(\boldsymbol{A})$. Then  $\sum\limits_{k=1}^d \lambda_k\langle x|A_kx\rangle\neq 0$. By Proposition \ref{chainrule}, we get
\begin{eqnarray*}
\partial f_{x,\boldsymbol{\lambda}}(\boldsymbol{A}) &=& T_{x,\boldsymbol{\lambda}}^*\partial g(T_{x,\boldsymbol{\lambda}}(\boldsymbol{A}))\\
&=& T_{x,\boldsymbol{\lambda}}^*\partial g\bigg(\sum\limits_{k=1}^d \lambda_k\langle x| A_kx\rangle\bigg)\\
&=& \left\{T_{x,\boldsymbol{\lambda}}^*\left( \frac{\sum\limits_{k=1}^d \lambda_k\langle x| A_kx\rangle}{\left|\sum\limits_{k=1}^d \lambda_k\langle x|A_kx\rangle\right|}\right)\right\}.
\end{eqnarray*}

Now $T_{x,\boldsymbol{\lambda}}^*: \C\rightarrow \mathbb M_n(\C)^d$ is the unique map satisfying 
\begin{equation}\label{chap5:11111}\tr [(T_{x,\boldsymbol{\lambda}}^*(z))^* \boldsymbol{C}]=\overline{z} \ T_{x,\boldsymbol{\lambda}}(\boldsymbol{C}).\end{equation}

If $T_{x,\boldsymbol{\lambda}}^*(z) = (T_1, \dots, T_d)$, then \eqref{chap5:11111} gives $$\sum\limits_{k=1}^d\tr(T_k^*C_k) = \sum\limits_{k=1}^d\overline{z}\lambda_k\langle x|C_k x\rangle.$$

This implies that for $z\in\C$, $T_{x,\boldsymbol{\lambda}}^*(z)=z\left(\overline{\lambda_1} x\bar{\tens{}}x, \overline{\lambda_2} x\bar{\tens{}}x, \dots, \overline{\lambda_d} x\bar{\tens{}}x\right)$.
So
$$\partial f_{x,\boldsymbol{\lambda}}(\boldsymbol{A})=\left\{ \frac{\sum\limits_{k=1}^d \lambda_k\langle x| A_kx\rangle}{\left|\sum\limits_{k=1}^d \lambda_k\langle x| A_kx\rangle\right|}\left(\overline{\lambda_1} x \bar{\tens{}}x, \overline{\lambda_2} x\bar{\tens{}} x, \dots, \overline{\lambda_d} x \bar{\tens{}}x\right)\right\}.$$

This gives
\begin{equation}\partial \omega(\boldsymbol{A})=\conv\left\{\frac{\sum\limits_{k=1}^d \lambda_k\langle x|A_kx\rangle}{\omega(\boldsymbol{A})}(\overline{\lambda_1} x\bar{\tens{}} x, \dots, \overline{\lambda_d} x\bar{\tens{}} x): (x,\boldsymbol{\lambda})\in J(\boldsymbol{A})\right\}.\label{chap5:75757}\end{equation}

For each $(x,\boldsymbol{\lambda})\in J(\boldsymbol{A})$, we have $$\left(\sum\limits_{k=1}^d|\langle x|A_kx\rangle|^2\right)^{1/2}\leq \omega(\boldsymbol{A}) = \left|\sum\limits_{k=1}^d \lambda_k\langle x|A_kx\rangle\right|\leq\left(\sum\limits_{k=1}^d|\langle x| A_kx\rangle|^2\right)^{1/2}.$$

The last inequality follows by the Cauchy-Schwarz inequality. Hence $$\left|\sum\limits_{k=1}^d \lambda_k\langle x| A_kx\rangle\right| = \left(\sum\limits_{k=1}^d\big|\langle x| A_kx\rangle\big|^2\right)^{1/2}\left(\sum\limits_{k=1}^d|\lambda_k|^2\right)^{1/2}.$$ By the condition of equality in the Cauchy-Schwarz inequality, there exists $\alpha\in\C$ such that $\left(\overline{\lambda_1}, \dots, \overline{\lambda_d}\right) = \alpha \left(\langle x| A_1x\rangle, \dots, \langle x| A_dx\rangle\right).$ This gives $\left|\alpha\right| = \left(\sum\limits_{k=1}^d|\langle x | A_kx\rangle|^2\right)^{-1/2}$. Thus for any $(x, \boldsymbol{\lambda})\in J(\boldsymbol{A})$, we get
\begin{align*}&\hspace{0.5cm} \dfrac{\sum\limits_{k=1}^d \lambda_k\langle x|A_kx\rangle}{\omega(\boldsymbol{A})}(\overline{\lambda_1} x\bar{\tens{}} x, \dots, \overline{\lambda_d} x\bar{\tens{}} x)\\
&= \overline{\alpha}\frac{\sum\limits_{k=1}^d |\langle x|A_kx\rangle|^2}{\omega(\boldsymbol{A})}\left(\alpha\langle x|A_1x\rangle x\bar{\tens{}}x, \dots, \alpha\langle x|A_dx\rangle x\bar{\tens{}}x\right)\\
&=  |\alpha|^2\frac{\sum\limits_{k=1}^d |\langle x|A_kx\rangle|^2}{\omega(\boldsymbol{A})}\left(\langle x|A_1x\rangle x\bar{\tens{}}x, \dots, \langle x|A_dx\rangle x\bar{\tens{}}x\right)\\
&= \frac{1}{\omega(\boldsymbol{A})}\left(\langle x|A_1x\rangle x\bar{\tens{}}x, \dots, \langle x|A_dx\rangle x\bar{\tens{}}x\right).\end{align*}
Substituting this in \eqref{chap5:75757}, we get \eqref{subdifferential}. This completes the proof.
\item[(ii)] This follows as an application of \eqref{chap5:7898} and \eqref{subdifferential}.
%Using \eqref{chap5:7898} and \eqref{subdifferential}, we have
%\begin{align*} \lim\limits_{t\rightarrow 0^+}\frac{\omega(\boldsymbol{A}+t\boldsymbol{B}) -\omega(\boldsymbol{A})}{t} &= \max\{\h\tr(\boldsymbol{C}^*\boldsymbol{B}) : C\in \partial \omega(A)\}\\
%& = \frac{1}{\omega(\boldsymbol{A})}\max\bigg\{\h\sum\limits_{l=1}^h \lambda_l\left(\sum\limits_{k=1}^d\tr\left((\langle x_l|A_kx_l\rangle x_l\bar{\tens{}}x_l)^* B_k\right)\right) : \\
%& \|x_l\|=1 \text{ and } \lambda_l\geq 0 \text{ for all } 1\leq l\leq h,\\ 
%& \sum\limits_{l=1}^h\lambda_h =1, \omega(\boldsymbol{A})=\left(\sum\limits_{k=1}^d\big|\langle x|A_kx\rangle\big|^2\right)^{1/2}\bigg\}\\
%&= \frac{1}{\omega(\boldsymbol{A})}\max_{\|x\|=1,\ \omega(\boldsymbol{A})=\left(\sum\limits_{k=1}^d|\langle x|A_kx\rangle|^2\right)^{1/2}}\h\sum\limits_{k=1}^d\tr\left(\overline{\langle x|A_kx\rangle} (x\bar{\tens{}}x) B_k\right)\\
%& = \frac{1}{\omega(\boldsymbol{A})}\max_{ \|x\|=1,\ \omega(\boldsymbol{A})=\left(\sum\limits_{k=1}^d|\langle x|A_kx\rangle|^2\right)^{1/2}}\h\sum\limits_{k=1}^d\overline{\langle x|A_kx\rangle}\langle x|B_kx\rangle.
%\end{align*}
\end{enumerate}\qed

%The right hand derivative of a continuous convex function $f: X\rightarrow \R$ and the subdifferential of $f$ are related as follows.
 %\begin{equation*}\lim\limits_{t\rightarrow 0^+} \dfrac{f(x+ty) - f(x)}{t} = \max\{\text{Re }v^*(w) : v^*\in \partial f(x)\}.\end{equation*}
%As an application of Theorem \ref{main}, we get the following expression for the right hand derivative of the joint numerical radius. For $\boldsymbol{A}, \boldsymbol{B}\in\mathbb M_n(\C)^d$, we have
%\begin{equation}\label{chap5:7898}\lim\limits_{t\rightarrow 0^+}\frac{\omega(\boldsymbol{A}+t\boldsymbol{B}) -\omega(\boldsymbol{A})}{t} = \frac{1}{\omega(\boldsymbol{A})}\max_{ \|x\|=1,\ \omega(\boldsymbol{A})=(\sum\limits_{k=1}^d|\langle x|A_kx\rangle|^2)^{1/2}}\h\sum_{k=1}^d\langle x|A_kx\rangle\overline{\langle x|B_kx\rangle}.\end{equation}

Using Theorem \ref{main}, we give the proof of Theorem \ref{100000} and Theorem \ref{101010}. The idea is similar to  \cite[Theorem 2.6]{2013} and \cite[Theorem 1]{2014}. 

\textit{Proof of Theorem \ref{100000}.}  %In the proof of Theorem \ref{101010}, if 
Without loss of generality, let $\boldsymbol{A}\neq\boldsymbol{0}$. Let $T_1:\C\rightarrow \mathbb M_n(\C)^d$ be the linear map defined as $T_1(\lambda)=\lambda\boldsymbol{B}$. Let $T_2:\C\rightarrow \mathbb M_n(\C)^d$ be the affine map defined as $T_2(\lambda)=T_1(\lambda)+\boldsymbol{A}$. Then \begin{equation*}
\omega(\boldsymbol{A}+\lambda\boldsymbol{ B}) \geq \omega(\boldsymbol{A}) \text{ for all } {\lambda}\in \C\text{ if and only if } 0\in \partial(\omega\circ T_2)(0).
\end{equation*}
By Proposition \ref{chainrule}, we get that
\begin{equation}
\omega(\boldsymbol{A}+\lambda\boldsymbol{B}) \geq \omega(\boldsymbol{A}) \text{ for all } \lambda\in \C \text{ if and only if } 0\in T_1^* \partial \omega(\boldsymbol{A}). \label{extremum1}
\end{equation}

The map $T_1^*: \mathbb M_n(\C)^d\rightarrow \C$ is given by $T_1^*(\boldsymbol{C}) = \sum\limits_{k=1}^d\tr(B_k^*C_k)$ for all $\boldsymbol{C}=(C_1, \dots, C_d)\in \mathbb M_n(\C)^d$. Therefore
\begin{align}T_1^*\partial\omega(\boldsymbol{A})  = \conv\bigg\{& \dfrac{1}{\omega(\boldsymbol{A})}\sum\limits_{k=1}^d\langle x|A_kx\rangle\overline{\langle x|B_kx\rangle} : \ \|x\|=1,\omega(\boldsymbol{A})=\left(\sum\limits_{k=1}^d\big|\langle x|A_kx\rangle\big|^2\right)^{1/2}\bigg\}.\label{chap5:8080}\end{align}
The result follows by substituting \eqref{chap5:8080} in \eqref{extremum1}. \qed
%, then we get the following characterization of orthogonality in $(\mathbb M_n(\mathbb C)^d, \omega(\cdot))$.

\textit{Proof of Theorem \ref{101010}.} Without loss of generality, let $\boldsymbol{A}\neq\boldsymbol{0}$. Let $T_1:\C^d\rightarrow \mathbb M_n(\C)^d$ be the linear map defined as $T_1(\boldsymbol{\lambda})=\boldsymbol{\lambda}\boldsymbol{B}$. Let $T_2:\C^d\rightarrow \mathbb M_n(\C)^d$ be defined as the affine map $L(\boldsymbol{\lambda})=T_1(\boldsymbol{\lambda})+\boldsymbol{A}$ for all $\boldsymbol{\lambda}\in\C^d$.
As in the proof of Theorem \ref{100000}, we obtain
\begin{equation}
\omega(\boldsymbol{A}+\boldsymbol{\lambda}\boldsymbol{B}) \geq \omega(\boldsymbol{A}) \text{ for all } \boldsymbol{\lambda}\in \C^d\text{ if and only if }\boldsymbol{0}\in T_1^* \partial \omega(\boldsymbol{A}). \label{extremum}
\end{equation}

The map $T_1^*: \mathbb M_n(\C)^d\rightarrow \C^d$ is given by $T_1^*(\boldsymbol{C}) = \left(\tr(B_1^*C_1), \dots, \tr(B_d^*C_d)\right)$ for all $\boldsymbol{C}=(C_1, \dots, C_d)\in \mathbb M_n(\C)^d$. Therefore
\begin{align}T_1^*\partial\omega(\boldsymbol{A})  = \conv\bigg\{& \dfrac{1}{\omega(\boldsymbol{A})}\left(\langle x|A_1x\rangle \overline{\langle x|B_1x\rangle}, \dots, \langle x|A_dx\rangle \overline{\langle x|B_dx\rangle}\right) : \ \|x\|=1,\nonumber\\
&\omega(\boldsymbol{A})=\left(\sum\limits_{k=1}^d\big|\langle x|A_kx\rangle\big|^2\right)^{1/2}\bigg\}.\label{chap5:8989}\end{align}
The result now follows by substituting \eqref{chap5:8989} in \eqref{extremum}. \qed

For $\boldsymbol{A}\in \mathscr B(\hc)$,  the joint operator norm of $\boldsymbol{A}$, denoted by $\|\boldsymbol{A}\|$, is defined as  $$\|\boldsymbol{A}\|=\sup
\left\{\left(\sum\limits_{k=1}^d\|A_kx\|^2\right)^{1/2} : x\in\mathbb C^n, \|x\|=1\right\}.$$ For the joint operator norm, an analogous result to Theorem \ref{100000} was proved in \cite[Corollary 3.4]{2013}. 

Let $X$ be a finite dimensional space and let $Y$ be a Banach space. Let $S_X$ denote the unit sphere of $X$. Let $C(S_X, Y)$ denote the space of continuous functions from $S_X$ to $Y$ with the supremum norm $\|\cdot\|_{\infty}$. For a bounded linear operator $T$ from $X$ to $Y$, define $\hat T: S_X\rightarrow Y$ as $\hat{T}(x) = T(x)$.  Then $\hat T\in C(S_X, Y)$ and $\|T\|=\|\hat{T}\|_{\infty}$.
%A bounded linear map $T$ from a finite dimensional space $X$ to a Banach space $Y$ can be identified with the continuous function from the unit sphere $S_X$ of $X$ to $Y$, defined by $\hat{T}(x) = T(x)$ for all $x\in S_X$.Then we have $\|T\|=\|\hat{T}\|_{\infty}$.
In particular, the space $\mathbb M_n(\C)^d$ equipped with the joint operator norm is isometrically isomorphic to a closed subspace of $C(S_{\C^n}, (\C^n)^d)$. In \cite{bilinear}, this identification was used to give an alternate proof of orthogonality to one dimensional subspaces in $\mathbb M_n(\C)$ given in \cite[Theorem 1]{Bhatia}. 
We use this identification to prove the following result for the joint operator norm, analogous to Theorem \ref{101010}.

\begin{theorem}\label{newproof} Let $\boldsymbol{A}, \boldsymbol{B}\in  \mathbb M_n(\C)^d$. Then \begin{equation}\|\boldsymbol{A+\lambda B}\|\geq \|\boldsymbol{A}\| \text{ for all } \boldsymbol{\lambda}\in \C^d\label{chap3:eqn1}\end{equation} if and only if 
there exist $h$ unit vectors $x_1,\ldots,x_h \in \C^n$ with $\|\boldsymbol{A}x_i\|=\|\boldsymbol{A}\|$ for all $1\leq i\leq h$  and there exist $h$ positive numbers $t_1,\ldots,t_h>0$ with $t_1+\cdots+t_h=1$ such that $\sum\limits_{i=1}^h t_i\langle A_kx_i|B_kx_i \rangle=0 \text{ for all } 1\leq k\leq d.$ 
Moreover, we have $1\leq h\leq 2d+1$.
\end{theorem}
\begin{proof} If $\boldsymbol{A}\in\mathbb C^d\boldsymbol{B}$, then the theorem holds trivially. %and $\boldsymbol{A}=0$ if it satisfies any of the conditions stated. 
So, without loss of generality, suppose $\boldsymbol{A}\notin\mathbb C^d\boldsymbol{B}$. By \cite[Theorem 1.6, p. 201]{Singer}, $\boldsymbol{A}$ is orthogonal to $\C^d\boldsymbol{B}$ if and only if there exist $h$ functionals $f_1, f_2, \dots, f_h\in ((\C^n)^d)^*$ of unit norm, $h$ unit vectors $x_1, \dots, x_h\in {\C^n}$ and $t_1, \dots, t_h>0$ with $\sum\limits_{i=1}^ht_i =1$ such that \begin{equation}\label{1}f_i(\boldsymbol{A}x_i)=\|\boldsymbol{A}\| \text{ for all } 1\leq i\leq h\end{equation} and \begin{equation}\label{2}\sum\limits_{i=1}^nt_if_i(\boldsymbol{\lambda B}x_i) = 0 \text{ for all } \boldsymbol{\lambda}\in\C^d.\end{equation} Moreover, $1\leq h\leq 2d+1$.
By the Riesz Representation Theorem, there exist unit vectors $y_1, \dots, y_h\in \mathbb C^n$ such that for $1\leq i\leq h$, $f_i(x) = \langle y_i| x\rangle$ for all $x\in\mathbb C^n$. So \eqref{1} is equivalent to the condition $\langle y_i|\boldsymbol{A}x_i\rangle =\|\boldsymbol{A}\|$. By the condition of equality in the Cauchy-Schwarz inequality, this is equivalent to saying that $y_i = \dfrac{1}{\|\boldsymbol{A}\|}\boldsymbol{A}x_i$. So $\|\boldsymbol{A}x_i\|=\|\boldsymbol{A}\|$. Thus \eqref{2} is equivalent to $\sum\limits_{i=1}^ht_i \langle \boldsymbol{A}x_i| \boldsymbol{\lambda B}x_i\rangle =0$ for all $ \boldsymbol{\lambda}\in\C^d$, that is, for $1\leq k\leq d$, $\sum\limits_{i=1}^h t_i\langle A_kx_i|B_kx_i \rangle=0$.
\end{proof}
%Let $I$ be the identity operator on $\Hil$. A special case of Theorem \ref{newproof} for $\boldsymbol{B}=(I, \dots, I)$ was proved in \cite[Theorem 1]{2021} and it was used to show a distance formula for tuples of opeators was also mentioned. 

%In the remaining discussion, we equip $\mathscr B(\hc)$ with the operator norm $\|\cdot\|$. In \cite[Theorem 2.8]{arxiv}, the following characterization is obtained for orthogonality to a subspace of operators.  Let $A\in\mathscr B(\hc)$ be such that $\|A\|=1$. Suppose the set $\{x\in\hc:\|Ax\|=\|A\|\}$ is the unit ball of a finite dimensional subspace $\hc_1$ of $\hc$ and $\|A\|_{\hc_1^{\perp}}<\|A\|$. Then for any subspace $\mathcal W$ of $\mathscr B(\hc)$, $A$ is orthogonal to $\mathcal W$ if and only if there exist unit vectors $x_1, \dots, x_h\in\hc_1$ with $\|Ax_i\|=\|A\|$ for all $1\leq i\leq h$ and there exist $t_1, \dots, t_h>0$ with $\sum\limits_{i=1}^ht_i =1$ such that  $\sum\limits_{i=1}^h t_i\langle Ax_i | Bx_i\rangle =0$ for all $B\in\mathcal W$. 

Let $\hk$ be a Hilbert space. Let $\mathscr B(\hc, \hk)$ denote the space of bounded linear operators from $\hc$ to $\hk$ and let $\|\cdot\|$ be the operator norm on it. The following characterization for orthogonality to a subspace of operators follows from \cite[Theorem 1.3]{new}.
%A generalization of \cite[Theorem 2.8]{arxiv} follows from \cite[Theorem 1.3]{new}. 
Let $A\in \mathscr B(\hc, \hk)$ be such that the set $\{x\in\hc:\|Ax\|=\|A\|\}$ is the unit ball of a finite dimensional subspace $\hc_1$ of $\hc$ and $\|A\|_{\hc_1^{\perp}}<\|A\|$. Let $\mathcal W\subseteq \mathscr B(\hc, \hk)$ be a subspace. Then $A$ is orthogonal to $\mathcal W$ if and only if there exist orthonormal vectors $x_1, \dots, x_h\in\hc_1$ with $\|Ax_i\|=\|A\|$ for all $1\leq i\leq h$ and there exist $t_1, \dots, t_h>0$ with $\sum\limits_{i=1}^ht_i =1$ such that  $\sum\limits_{i=1}^h t_i\langle Ax_i | Bx_i\rangle =0$ for all $B\in\mathcal W$. A special case of this when $\hk=\hc$ and $\|A\|=1$ was proved in \cite[Theorem 2.8]{arxiv}. When $\mathcal W$ is a finite dimensional subspace, we can obtain an upper bound on $h$ as $2\dim(\mathcal W)+1$. The proof of this fact follows along the lines of proof of Theorem \ref{newproof} above.  
%The below theorem gives a characterization if orthogonality in $\mathscr B(\hc, \hk)$. The below theorem show that the condition $\|A\|=1$ in  \cite[Theorem 2.8]{arxiv} is not necessary for the result to hold, also we find a bound on number of unit vectors in the characterization. 
%\begin{theorem}\label{bound} Let $A\in\mathscr B(\hc, \hk)$ be such that the set $\{x\in\hc:\|Ax\|=\|A\|\}$ is the unit ball of a finite dimensional subspace $\hc_1$ of $\hc$ and $\|A\|_{\hc_1^{\perp}}<\|A\|$. Let $\mathcal W$ be a finite dimensional subspace of $\mathscr B(\hc, \hk)$. Then $A$ is orthogonal to $\mathcal W$ in  $(\mathscr B(\hc, \hk), \|\cdot\|)$ if and only if there exist unit vectors $x_1, \dots, x_h\in\hc_1$ with $\|Ax_i\|=\|A\|$ for all $1\leq i\leq h$ and there exist $t_1, \dots, t_h>0$ with $\sum\limits_{i=1}^ht_i =1$ such that  $\sum\limits_{i=1}^h t_i\langle Ax_i | Bx_i\rangle =0$ for all $B\in\mathcal W$. Moreover, $1\leq h\leq 2\dim(\mathcal W)+1$.\end{theorem}

Since the vectors $x_1, \dots, x_h$ can be chosen to be linearly independent, $\text{dim}(\hc)$ is also a bound on $h$. Theorem 1 of \cite{2014} and Theorem 8.4 of \cite{Zietak2} are special cases of this. In both the papers, the bound on $h$ was shown to be $dim(\hc)$ and we have been able to find a better bound on $h$. When $\mathcal W$ is a one dimensional subspace, a characterization of orthogonality was first proved in \cite[Lemma 2.2]{magajna}. It was motivated by the proof of \cite[Lemma 9.14]{Davidson}.  An alternate proof of this can be found in \cite[Remark 3.1]{Bhatia}. For a detailed survey on Birkhoff-James orthogonality to subspaces and its applications, see \cite{2019, 2020, 2022} and the references therein.

\section{Remarks}\label{Remarks}

\textbf{Remark 1}. Let $X$ be a reflexive Banach space and let $Y$ be a Banach space. Let $\mathscr K(X, Y)$ be the space of  compact operators from $X$ to $Y$ with the operator norm. For $x\in X$ and a subspace $W$ of $X$, let $\dist(x, W)=\inf\{\|x-w\|:w\in W\}$. Theorem \ref{newproof} also holds for $A\in\mathscr B(X, Y)$ such that $\dist(A, \mathscr K(X, Y))<\|A\|$. This can be seen from \cite[Lemma 3.1]{Wojcik} and the proof of Corollary \ref{101010}. An expression for the subdifferential set of the norm function in $\mathscr B(X, Y)$ for a reflexive Banach space $X$ was also obtained in \cite[Theorem 3.2]{Wojcik}. Another recent work on the subdifferential of the norm function is \cite{rao}.

%\textbf{Remark 2}. By the Haudorff-Toeplitz theorem and Theorem \ref{bound}, we get an alternative proof \cite[Theorem 1.1]{Bhatia} that gives a characterization of orthogonality of an element to one dimensional subspace in the space of matrices. A characterization of orthogonality of an element to one dimensional subspace in the space of operators on a Hilbert space was also obtained in \cite[Remark 3.1]{Bhatia} and \cite[Lemma2]{Radii}. We would like to mention that this result was first proved in \cite[Lemma 2.2]{magajna}, which seem to be missing from the literature on Birkhoff-James orthogonality. It was motivated by the proof of the \cite[Lemma 9.14]{Davidson}. And it is also interesting to note that \cite[Lemma 9.15]{Davidson} gives an alternate proof of  \cite[Theorem 4]{derivation}, and the idea of the proof of \cite[Theorem 1.2]{Bhatia} was along the same lines.

%\begin{remark} A natural thing to ask will be: Is it possible to find $h=1$ in the Theorem \ref{100000} or Theorem \ref{101010} or Theorem \ref{newproof}? And the answer is no, in general. A counterexample that $h$ may not be $1$ in Theorem \ref{newproof} was given in \cite{2021}. But in 
%\end{remark}

\textbf{Remark 2}. Birkhoff-James orthogonality is closely related to the notion of norm parallelism. In a normed space $(X,\|\cdot\|)$, an element $x$ is said to be \emph{norm parallel} to another element $y$ in $(X,\|\cdot\|)$ if there exists $\lambda\in\mathbb C$ such that $|\lambda|=1$ and $\|x+\lambda y\|=\|x\|+\|y\|$. Let $\boldsymbol{A}, \boldsymbol{B}\in\mathbb M_n(\C)^d$. Then by \cite[Theorem 2.4]{normparallel} and Theorem \ref{100000}, we get that $\boldsymbol{A}$ is norm parallel to $\boldsymbol{B}$ in $(\mathbb M_n(\C)^d, \omega(\cdot))$ if and only if there exists a unit vector $x\in\C^n$ such that $\big|\sum_{k=1}^d \overline{\bra x|B_kx\ket}\bra x|A_kx\ket\big|=\omega(\boldsymbol{A})\omega(\boldsymbol{B})$. The same characterization also holds for $\boldsymbol{A}, \boldsymbol{B}\in\mathscr B(\hc, \hc^d)$ for a Hilbert space $\hc$. The proof can be done along the lines of the proof of \cite[Theroem 2.2]{radiusparallel}.

\end{document}